\documentclass[12pt]{article}

\usepackage{amssymb,amsfonts,amscd,amsthm}
\usepackage[all,arc]{xy}
\usepackage[mathlines]{lineno} 
\usepackage{soul} 
\usepackage{enumerate, bbm}
\usepackage{mathrsfs}
\usepackage{tikz}
\usepackage{mathtools}
\usepackage{hyperref}
\usepackage{color} 
\usepackage{graphicx}
\usepackage{enumitem} 
\graphicspath{ {figs/} }
\usepackage{caption}
\usepackage{subcaption}
\usepackage{pgf}
\newtheorem{thm}{Theorem}[section]
\newtheorem{cor}[thm]{Corollary}
\newtheorem{prop}[thm]{Proposition}
\newtheorem{lem}[thm]{Lemma}
\newtheorem{claim}[thm]{Claim}

\newtheorem{prob}[thm]{Problem}
\newtheorem{conj}[thm]{Conjecture}

\theoremstyle{definition}

\newtheorem{rem}[thm]{Remark}
\newtheorem{obs}[thm]{Observation}

\newtheoremstyle{TheoremNum}
    {\topsep}{\topsep}
    {\itshape}
    {}
    {\bfseries}
    {.}
    { }
    {\thmname{#1}\thmnote{ \bfseries #3}}
\theoremstyle{TheoremNum}
\newtheorem{thmn}{Theorem}
\newtheorem{propn}{Proposition}
\newtheorem{corn}{Corollary}

\hypersetup{
	colorlinks,
	citecolor=blue,
	filecolor=blue,
	linkcolor=blue,
	urlcolor=blue,
	linktocpage
}

\setenumerate[1]{label=\thesection.\arabic*.} 
\setenumerate[2]{label*=\arabic*.} 
\setlist[enumerate]{itemsep=2ex, topsep=2ex} 
\setlist[itemize]{itemsep=2ex, topsep=2ex}

\newcommand{\Z}{\mathbb{Z}}
\newcommand{\half}{\frac{1}{2}}

\newcommand{\sm}{\setminus}
\newcommand{\sub}{\subseteq}
\newcommand{\Del}{\Delta}
\newcommand{\del}{\delta}

\newcommand{\Om}{\Omega}

\newcommand{\ol}[1]{\overline{#1}}
\newcommand{\tr}[1]{\textrm{#1}}

\newcommand{\f}[2]{\frac{#1}{#2}}
\newcommand{\floor}[1]{\left\lfloor #1\right\rfloor}

\newcommand{\pr}[1]{\text{Pr}\left[#1\right]}
\newcommand{\thresh}[1]{p(#1)}
\DeclareMathOperator{\zfs}{ZFS}
\newcommand{\Bzf}[2]{B_{#1}(#2) \in \zfs(#2)} 

\title{Zero Forcing with Random Sets}
\author{Bryan Curtis\footnote{Department of Mathematics, Iowa State University {\tt bcurtis1@iastate.edu}. This work is partially supported by NSF grant 1839918.}
\and Luyining Gan\footnote{Department of Mathematics and Statistics, University of Nevada Reno {\tt lgan@unr.edu}. This work is supported by AMS-Simons Travel Grant 2022-2024.}
\and Jamie Haddock\footnote{Department of Mathematics, Harvey Mudd College {\tt jhaddock@g.hmc.edu}. This work is supported by NSF DMS award \#2211318.}
\and Rachel Lawrence\footnote{Department of Electrical Engineering and Computer Science, University of California, Berkeley {\tt rlaw@berkeley.edu}.}
\and Sam Spiro\footnote{Department of Mathematics, UCSD {\tt sspiro@ucsd.edu}. This material is based upon work supported by the National Science Foundation Graduate Research Fellowship under Grant No. DGE-1650112.}}
\date{}

\begin{document}
\maketitle
\begin{abstract}
Given a graph $G$ and a real number $0\le p\le 1$, we define the random set $B_p(G)\sub V(G)$ by including each vertex independently and with probability $p$.  We investigate the probability that the random set $B_p(G)$ is a zero forcing set of $G$.
In particular, we prove that for large $n$, this probability for trees is upper bounded by the corresponding probability for a path graph.
Given a minimum degree condition, we also prove a conjecture of Boyer et.\ al.\ regarding the number of zero forcing sets of a given size that a graph can have.
\end{abstract}

\noindent{\bf Keywords:}
 Random set zero forcing, Zero forcing polynomial, Random graphs, Graph threshold.



%
%
\section{Introduction}
Let $G$ be a graph with vertex set $V(G)$ initially colored either blue or white. If $u$ is a blue vertex of $G$ and the neighborhood $N_G(u)$ of $u$ contains exactly one white vertex $v$, then we may change the color of $v$ to blue. This iterated procedure of coloring a graph is called zero forcing. A \emph{zero forcing set} $B$ is a subset of vertices of $G$ such that if $G$ initially has all of the vertices of $B$ colored blue, then the zero forcing process can eventually color all of $V(G)$ blue.  We let $\zfs(G)$ denote the set of all zero forcing sets of $G$. The \emph{zero forcing number} $Z(G)$ is the minimum cardinality of a zero forcing set in $G$; that is, $Z(G) = \min_{B \in \zfs(G)} |B|$. The zero forcing process was first introduced by Burgarth and Giovannetti~\cite{BG07}, and the zero forcing number was introduced by an AIM's research group~\cite{AIM08} as a bound for the maximum nullity of a graph $G$.

\begin{figure}[h!]
    \centering
    \begin{subfigure}[b]{0.3\textwidth}
         \centering
         \includegraphics[scale=0.5]{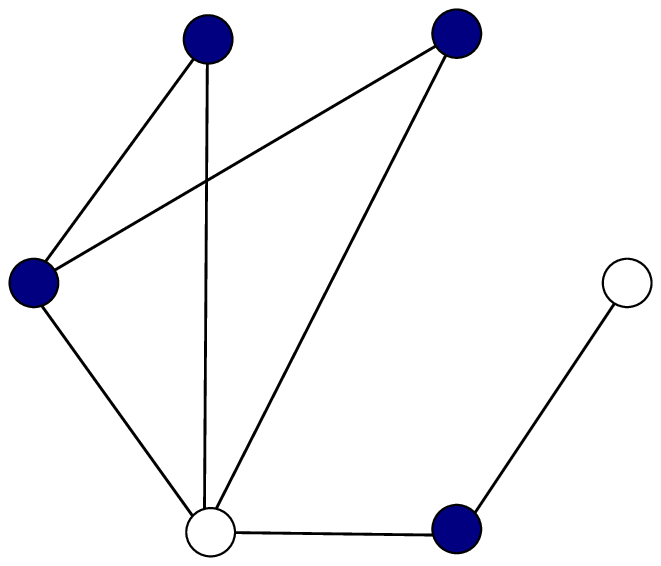}
         \caption{}
         \label{fig:zf_example1}
     \end{subfigure}
     \hfill
     \begin{subfigure}[b]{0.3\textwidth}
         \centering
         \includegraphics[scale=0.5]{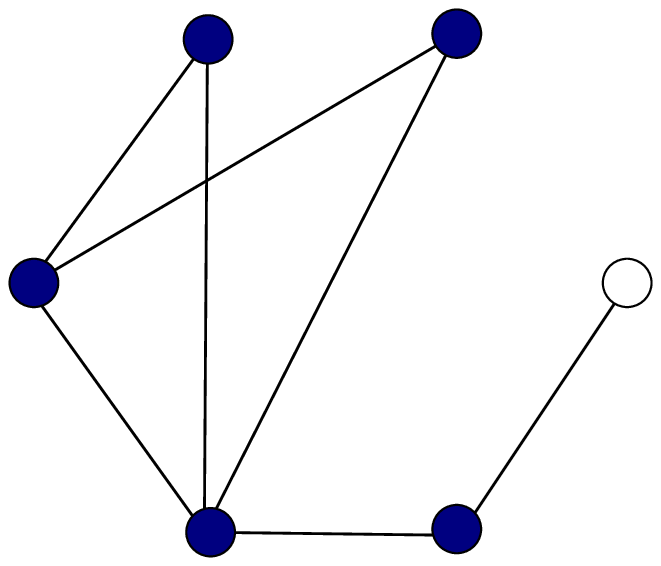}
         \caption{}
         \label{fig:zf_example2}
     \end{subfigure}
     \hfill
     \begin{subfigure}[b]{0.3\textwidth}
         \centering
         \includegraphics[scale=0.5]{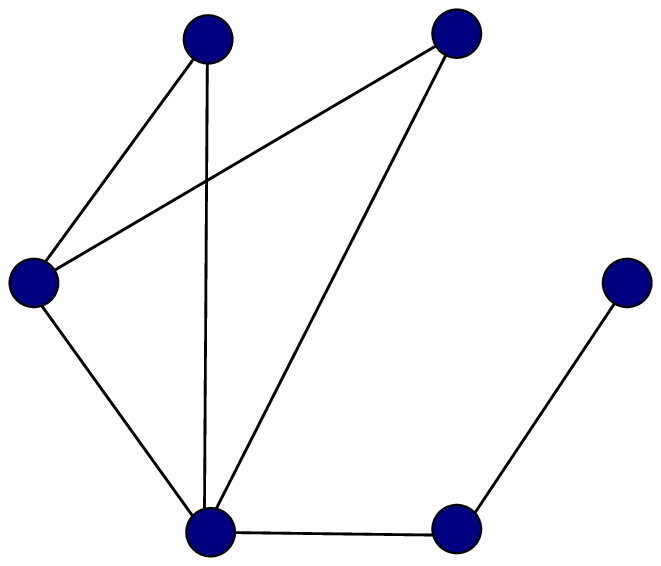}
         \caption{}
         \label{fig:zf_example3}
     \end{subfigure}
    \caption{An example of zero forcing on a graph.}
    \label{fig:zf_example}
\end{figure}

In this paper we consider a randomized version of the zero forcing process. There are seemingly two natural ways to define such a random process: one can either use a deterministic set of blue vertices $B$ together with forces that occur randomly, or one can use a random set of vertices together with deterministic forces. The process known as \textit{probabilistic zero forcing}, which was introduced by Kang and Yi~\cite{KY13}, is of the former type and is by now well studied, see for example~\cite{CCG20, GH18, NS21,EMP21, HS20}.  In this paper we introduce a process of the latter type which we call \textit{random set zero forcing}.

\subsection{Main Results}
Given a graph $G$ and real number $0\le p\le 1$, we define the random set $B_p(G)\sub V(G)$ by including each vertex of $G$ independently and with probability $p$.  For example, $B_1(G)=V(G)$, $B_0(G)=\emptyset$, and $B_{1/2}(G)$ is equally likely to be any subset of $V(G)$.

The central problem we wish to ask is, given $G$ and $p$, what is (approximately) the probability that $B_p(G)$ is a zero forcing set of $G$?  For example, one general bound we can prove is the following.
\begin{thm}\label{thm:degree}
Let $G$ be an $n$-vertex graph with minimum degree at least $\del\ge 1$.  For all $p$, we have
\[\Pr[\Bzf{p}{G}]\le \del n p^\del.\]
\end{thm}
In fact, we prove a slightly stronger version of this theorem that holds for graphs with ``few'' vertices of degree less than $\del$; see Theorem~\ref{thm:degreeLower}. Theorem~\ref{thm:degree} can be viewed as a probabilistic analog to the basic fact that $Z(G)\ge \del$ if $G$ is a graph with minimum degree $\del$.

Many other results from classical zero forcing also have probabilistic analogs for random set zero forcing.  For example, it is straightforward to show that if $G$ is an $n$-vertex graph, then $Z(G)\le Z(\ol{K_n})$ with equality if and only if $G=\ol{K_n}$.  In the random setting, it is also easy to show the analogous result that for all $n$-vertex graphs $G$ and $0\le p\le 1$, we have
\[\Pr[\Bzf{p}{G}]\ge \Pr[\Bzf{p}{\ol{K_n}}],\]
with equality holding if and only if either $p\in \{0,1\}$ or $G=\ol{K_n}$.

In the classical setting, it is well known that amongst $n$-vertex graphs, the path $P_n$ is the unique graph with the smallest zero forcing number.  We conjecture that an analog of this result holds in the random setting.
\begin{conj}\label{conj:path}
If $G$ is an $n$-vertex graph and $0\le p\le 1$, then
\[\Pr[\Bzf{p}{G}]\le \Pr[\Bzf{p}{P_n}],\]
with equality holding if and only if either $p\in \{0,1\}$ or $G=P_n$.
\end{conj}
While we do not prove this conjecture in full, we provide some partial results; in particular, we prove the conjecture when restricted to trees and with $n$ sufficiently large.
\begin{thm}\label{thm:tree}
If $T$ is an $n$-vertex tree with $n$ sufficiently large, then for all $0\le p\le 1$,
\[\Pr[\Bzf{p}{T}]\le \Pr[\Bzf{p}{P_n}],\]
with equality holding if and only if either $p\in \{0,1\}$ or $T=P_n$.
\end{thm}

For many graphs $G$, it will happen that there exists a $p$ such that $B_{p'}(G)$ is very unlikely to be a zero forcing set if $p'$ is much smaller than $p$, and that $B_{p'}(G)$  is very likely to be zero forcing if $p'$ is much larger than $p$, see for example Figure \ref{fig:EmpiricalGraph}.  To capture what this value of $p$ is, we define the \emph{threshold probability} $\thresh{G}$ to be the unique $p$ such that $\Pr[\Bzf{p}{G}]=\half$, and it is not difficult to see that $\thresh{G}$ is well defined.  This definition is motivated by the study of thresholds in random graphs, which is one of the fundamental topics in probabilistic combinatorics (see, for example \cite{FK16}).

\begin{figure}[h!]
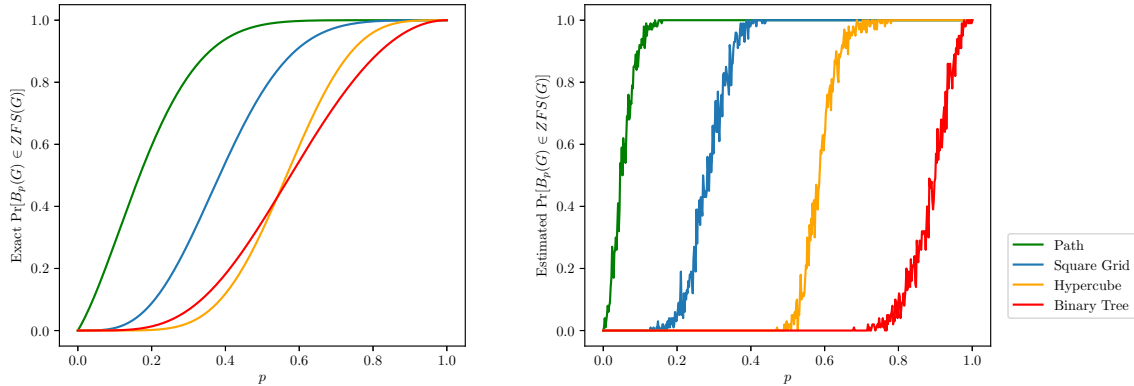

     \hspace{20pt}\resizebox{0.50\textwidth}{!}{\input{figs/empirical_4graphs_16vertices_formatting.pgf}}
     \hspace{-40pt}\resizebox{0.50\textwidth}{!}{\input{figs/empirical_4graphs_256vertices_formatting3.pgf}}
    \caption{Exact (left) and Monte Carlo estimates (right) of $\Pr[\Bzf{p}{G}]$ for the path, square grid, hypercube, and left-complete binary tree graphs on 16 and 256 vertices respectively.}
    \label{fig:EmpiricalGraph}
\end{figure}

We note that if Conjecture~\ref{conj:path} were true, then in particular we would have $\thresh{G}>\thresh{P_n}$ for all $n$-vertex graphs $G\ne P_n$; we subsequently prove that this is true up to a constant factor. Here and throughout the paper we use standard asymptotic notation, which is recalled in Subsection~\ref{Sec:notation}.
\begin{thm}\label{thm:pathThreshold}
If $G$ is an $n$-vertex graph, then
\[\thresh{G}=\Om(\thresh{P_n})=\Om(n^{-1/2}).\]
\end{thm}
In essence this result says that a random set of much less than $n^{1/2}$ vertices of any $n$-vertex graph $G$ is very unlikely to be a zero forcing set for sufficiently large $n$.

Conjecture~\ref{conj:path} can be viewed as a weakened version of a conjecture involving the number of zero forcing sets of a given size.  To this end, we observe that if $G$ is an $n$ vertex graph and $z(G;k)$ is the number of zero forcing sets of $G$ of size $k$, then
\begin{equation}\label{zf poly coeff}
\Pr[\Bzf{p}{G}]=\sum_{k = 1}^{n} z(G;k)p^k(1-p)^{n-k}.
\end{equation}
The notation $z(G;k)$ follows that of Boyer et.\ al.\ in \cite{boyer2019zero} who introduced the study of zero forcing polynomials and found many explicit formulas for $z(G;k)$, including:
\begin{equation}\label{eqn:zfs(path)}
z(P_n;k) = \binom{n}{k} - \binom{n-k-1}{k}.
\end{equation}
Observe that, by (\ref{zf poly coeff}), Conjecture~\ref{conj:path} is a weakened version of the following conjecture.
\begin{conj}[\cite{boyer2019zero}]\label{conj:pathCount}
If $G$ is an $n$-vertex graph, then for all $k$, \[z(G;k)\le z(P_n;k)={n\choose k}-{n-k-1\choose k}.\]
\end{conj}
It was shown in \cite{boyer2019zero} that Conjecture \ref{conj:pathCount} holds whenever $G$ contains a Hamiltonian path, but other than this very little is known. By extending our proof of Theorem~\ref{thm:degree}, we prove Conjecture~\ref{conj:pathCount} whenever $k$ is sufficiently small, as a function of the minimum degree of $G$.

\begin{prop}\label{prop:degree}
If $G$ is an $n$-vertex graph with minimum degree $\del\ge 3$, then for all $k\le (2\del)^{-1/\del}n^{1-1/\del}$ we have $z(G;k)\le z(P_n;k)$.
\end{prop}
We additionally show that this implies Conjecture~\ref{conj:pathCount} whenever $G$ has sufficiently large minimum degree.
\begin{cor}\label{cor:degree}
If $G$ is an $n$-vertex graph with minimum degree $\del\ge \log_2(n)+2\log_2\log_2(n)$, then  $z(G;k)\le z(P_n;k)$ for all $k$.
\end{cor}

\subsection{Organization and Notation}\label{Sec:notation}
This paper is organized as follows. In Section~\ref{Sec:examples}, we provide specific bounds on the threshold probability for several families of graphs and some graph operations. In Section~\ref{Sec:degree}, we provide a general bound on the probability that $B_p(G)$ is zero forcing given the minimum degree.
In Section~\ref{Sec:path}, we prove that the threshold probability for an $n$-vertex graph $G$ is $\Om(n^{-1/2})$. In Section~\ref{Sec:tree}, we prove that  amongst trees on sufficiently many vertices, paths have the largest probability of $B_p(G)$ being a zero forcing set. We conclude with some remarks and open questions in Section~\ref{Sec:remark}.

Throughout we use standard asymptotic notation. Let $f(n)$ and $g(n)$ be functions from the non-negative integers to the reals. We write $f(n) = o(g(n))$ if $\lim_{n\to\infty}{f(n)}/{g(n)} = 0$, and $f(n) = O(g(n))$ if there exists a $C > 0$ such that $f(n) \leq Cg(n)$ for all sufficiently large $n$. We write $f(n) = \Omega(g(n))$ if $g(n) = O(f(n))$, and $f(n) = \Theta(g(n))$ if both $f(n) = O(g(n))$ and $f(n) = \Omega(g(n))$.  We write, for example, $\Theta_k(g(n))$ if the implicit constants depend on $k$.

%
%
\section{Threshold Probabilities for Families of Graphs}\label{Sec:examples}
In this section we analyze the threshold probability for some families of graphs. These results provide an introduction to the various probabilistic and zero forcing arguments made throughout this paper, and also highlight some of the interesting properties of random set zero forcing.

We make use of the following inequalities throughout this section. Recall that
\begin{equation}\label{ineq:expvslin}
   1-x \le e^{-x}
\end{equation}
for all real values $x$, and
\begin{equation}\label{ineq:cexpvsclin}
   \left(1-\frac{c}{n}\right)^n \ge 1-c
\end{equation}
for $|c| \le n$ and $n \ge 1$.

 Let $G$ be a graph on $n\geq 2$ vertices with no isolated vertices. It is well known that every subset of $V(G)$ of size $n-1$ is a zero forcing set of $G$, and that $Z(G) = n-1$ if and only if $G=K_n$, the complete graph on $n$ vertices (see, for example, \cite{hogben2022inverse}). These observations imply the following.

\begin{prop}\label{prop:Kn thresh}
If $G$ is a graph on $n$ vertices with no isolated vertices, then $p(G) \leq p(K_n)$.  Moreover, $p(K_n) = 1 - \Theta(n^{-1})$.
\end{prop}
\begin{proof}
The result is immediate for $n=1$, so assume $n\geq2$. Define
\[
f(p)= n(1-p)p^{n-1} + p^n,
\]
which is the probability that $B_p(G)$ contains at least $n-1$ vertices.  Since every subset of $V(G)$ of size $n-1$ is a zero forcing set, we have for $p\in [0,1]$,
\[
\pr{\Bzf{p}{G}} \geq f(p) = \pr{\Bzf{p}{K_n}},
\]
where this equality used that a set $S$ is a zero forcing set of $K_n$ if and only if $|S|\geq n - 1$. Since $\pr{\Bzf{p}{G}}$ and $\pr{\Bzf{p}{K_n}}$ are increasing functions of $p$, we conclude that $p(G) \leq p(K_n)$.

We now prove the asymptotic result.  Let $p = 1 - c/n$, where $c\leq n$ is positive. By (\ref{ineq:cexpvsclin}),
\[
f(p)\ge p^n\ge 1-c,
\]
which implies $f(p)>1/2$ if $p>1-\frac{1}{2n}$.  Similarly, by (\ref{ineq:expvslin}),
\[
f(p)=c (1-c/n)^{n-1}+(1-c/n)^n\le c e^{-c+c/n}+e^{-c}\le c e^{-c/2}+e^{-c},
\]
where the last inequality holds since $n\geq2$. Thus $f(p)<1/2$ for $n\geq 5$ and $p<1-\frac{5}{n}$. We conclude $\thresh{K_n} = 1 - \Theta(n^{-1})$.
\end{proof}

By allowing for isolated vertices, it is possible for an $n$-vertex graph to have a higher threshold probability than $K_n$. Indeed, it is easy to show that the largest threshold probability amongst all graphs on $n$ vertices is $p(nK_1)$, where $nK_1$ denotes the graph on $n$ isolated vertices.  In fact, we can use Observation~\ref{obs: disjoint union} stated below to give the exact result $p(nK_1) = 2^{-1/n}$.  Moreover, Observation~\ref{obs: disjoint union} allows us to reduce our focus to connected graphs.
\begin{obs}\label{obs: disjoint union}
Let $G$ be the disjoint union of the graphs $G_1$ and $G_2$. Then
\[
\pr{B_p(G)\in \zfs(G)} = \pr{B_p(G_1)\in \zfs(G_1)}\cdot \pr{B_p(G_2)\in \zfs(G_2)}.
\]
\end{obs}

While it is straightforward to determine the graphs with the largest threshold probabilities, the analogous problem for smallest thresholds appears much harder. Intuitively, the path graph $P_n$ is a natural candidate for the minimizer, since it is known that $P_n$ is the unique $n$-vertex graph with zero forcing number 1.

Towards this end, we establish the order of magnitude of $p(P_n)$.  By convention, we assume the vertices $v_1,\ldots,v_n$ of $P_n$ are in \textit{path order}, i.e., the edges of $P_n$ are $v_iv_{i+1}$ for $1\leq i \leq n-1$. Note that $S\subseteq V(P_n)$ is a zero forcing set if and only if $S$ contains an endpoint or $S$ contains two consecutive vertices (see Figure~\ref{fig:zf_path}).

\begin{figure}[h!]
    \centering
    \begin{subfigure}[b]{0.3\textwidth}
         \centering
         \includegraphics[scale=0.4]{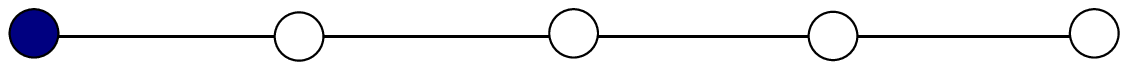}
         \label{fig:zf_path1}
     \end{subfigure}
     \hfill
     \begin{subfigure}[b]{0.3\textwidth}
         \centering
         \includegraphics[scale=0.4]{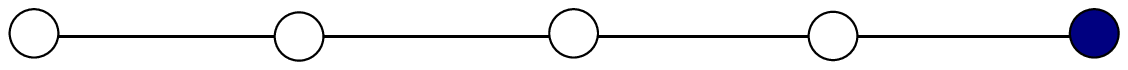}
         \label{fig:zf_path2}
     \end{subfigure}
     \hfill
     \begin{subfigure}[b]{0.3\textwidth}
         \centering
         \includegraphics[scale=0.4]{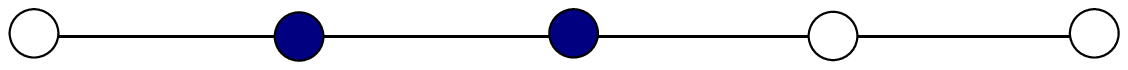}
         \label{fig:zf_path3}
     \end{subfigure}
    \caption{Three zero forcing sets for $P_5$}
    \label{fig:zf_path}
\end{figure}

\begin{prop}\label{prop:p(P_n)}
The threshold probability of the path on $n$ vertices satisfies \[p(P_n) = \Theta(n^{-1/2}).\]
\end{prop}
\begin{proof}
Let $v_1,\ldots,v_n$ denote the vertices of $P_n$ (in path order). Define the random variable $X$ to be the number of indices $i \in \{1, 2, \cdots, n-1\}$ such that $v_i, v_{i+1} \in B_p(P_n)$.  Markov's inequality yields
\[
\pr{X \ge 1} \le \mathbb{E}[X] = (n-1)p^2.
\]
Since $\Bzf{p}{P_n}$ if and only if either $X\ge 1$ or at least one of $v_1,v_n\in B_p(P_n)$, a union bound now implies
\[
\pr{\Bzf{p}{P_n}} \le (n-1)p^2 + 2p.
\]
This quantity is less than $1/2$ provided $p = cn^{-1/2}$ for any $c < 1/4$. Thus $\thresh{P_n}=\Om(n^{-1/2})$.

Next, for $i\in\{1,2,\ldots,n-1\}$, let $A_i$ be the event that $v_i, v_{i+1} \in B_p(P_n)$, and define $\displaystyle A = \bigcup_{i \text{ odd}} A_i$. Then,
\begin{align*}
    \pr{\Bzf{p}{P_n}} &\ge \pr{A} = 1 - \prod_{i \text{ odd}} (1 - \pr{A_i})\\
    &= 1 - (1-p^2)^{\lfloor (n-1)/2 \rfloor} \ge 1 - e^{-p^2 \lfloor (n-1)/2 \rfloor},
\end{align*}
where the first equality follows from the fact that these events are independent, and the  last inequality follows from (\ref{ineq:expvslin}).  This probability will be greater than $1/2$ for $p = Cn^{-1/2}$ with $C$ sufficiently large.  We conclude that $\thresh{P_n}=\Theta(n^{-1/2})$.
\end{proof}

A similar bound holds for the cycle graph $C_n$.
\begin{prop}\label{prop:p(C_n)}
The threshold probability of the cycle on $n$ vertices satisfies
\[p(C_n) = \Theta(n^{-1/2}).\]
\end{prop}
\begin{proof}
Observe that $S$ is a zero forcing set of $C_n$ if and only if $S$ contains a pair of consecutive vertices. Thus, using an argument similar to the proof of Proposition \ref{prop:p(P_n)}, we find
\[
1 - (1-p^2)^{\lfloor (n-1)/2 \rfloor} \le \pr{\Bzf{p}{C_n}} \le np^2.
\]
We conclude that $\thresh{C_n}=\Theta(n^{-1/2})$.
\end{proof}

It is perhaps intuitive that $p(C_n)\approx p(P_n)$ since $C_n$ can be formed from $P_n$ by adding a single edge.  However, there are examples where this intuition fails. Indeed, let $v_1,\ldots,v_n$ be the vertices of $P_n$, and let $R_n$ denote the graph obtained from $P_n$ by adding the edge $v_1v_3$ (see Figure~\ref{fig:R_n}).

\begin{figure}[h!]
    \centering
    \begin{tikzpicture}
        \node {\includegraphics[width=0.5\textwidth]{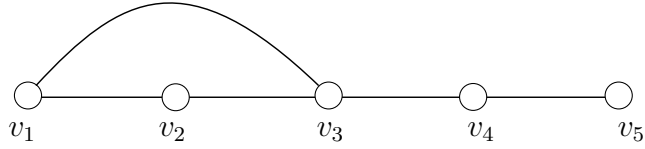}};
        \node at (-4, -1) {$v_1$};
        \node at (-2, -1) {$v_2$};
        \node at (0.1, -1) {$v_3$};
        \node at (2.1, -1) {$v_4$};
        \node at (4.1, -1) {$v_5$};
    \end{tikzpicture}
    \caption{The triangle with pendant path on five vertices, $R_5$.} 
    \label{fig:R_n}
\end{figure}

\begin{prop}\label{prop:p(R_n)}
The threshold probability of the triangle with a pendant path on $n$ vertices satisfies \[p(R_n) = \Theta(1).\]
\end{prop}
\begin{proof}
Note that any zero forcing set of $R_n$ must contain either $v_1$ or $v_2$.  Thus,
\[
\pr{\Bzf{p}{R_n}} \le \pr{v_1\in B_p(R_n) \text{ or } v_2 \in B_p(R_n)} \le 2p.
\]
This implies $\thresh{R_n} \in [1/4,1]$, and hence $\thresh{R_n} = \Theta(1)$.
\end{proof}

\begin{rem}
In the deterministic setting, it is well known that the zero forcing number $Z(G)$ of a graph $G$ changes by at most one if a single edge or vertex is removed from $G$.  This is far from true for $p(G)$.  Indeed, recall that $p(P_n)=\Theta(n^{-1/2})$.  Let $P'_n$ be obtained by deleting the edge $v_1v_2$.  Since $P'_n$ has $K_1$ as a connected component, by Observation~\ref{obs: disjoint union} we have $p(P'_n)\ge p(K_1)=\half$.  A similar result holds if one deletes $v_2$.

We next show that deleting edges or vertices can dramatically decrease $p(G)$. Consider the triangle with pendant path $R_n$, which has $\thresh{R_n} = \Theta(1)$.  If one deletes $v_1$, then the resulting graph is $P_{n-1}$, which has $\thresh{P_{n-1}} = \Theta(n^{-1/2})$.  If one deletes the edge $v_1v_3$, then the resulting graph is $P_n$, which has $\thresh{P_n} = \Theta(n^{-1/2})$.
\end{rem}

We end with a table that summarizes our main results of this section and states (without proof) other results that can be obtained using our methods.

\begin{table}
\caption{Thresholds for graph families.}\label{tab:examples}
\begin{center}
\begin{tabular}{ l | l | l }
Family & Description & Threshold Probability\\ \hline
$K_n$ & Complete graph on $n$ vertices: Proposition \ref{prop:Kn thresh} & $1-\Theta(n^{-1})$\\
$nK_1$ & Graph on $n$ isolated vertices & $ 2^{-1/n}$\\

$K_{n_1, \cdots, n_k}$ & Complete multipartite graph &  $1-\Theta_k(\min_i\{n_i^{-1}\})$\\
$P_n$ & Path on $n$ vertices: Proposition \ref{prop:p(P_n)} &  $\Theta(n^{-1/2})$\\
$C_n$ & Cycle on $n$ vertices: Proposition \ref{prop:p(C_n)} &  $\Theta(n^{-1/2})$\\

$W_n$ & Wheel on $n$ vertices: $C_{n-1} +  K_1$ & $\Theta(n^{-1/3})$

\end{tabular}
\end{center}
\end{table}

\section{Bounds Using Degrees}\label{Sec:degree}
In this section we prove bounds on the probability that $B_p(G)$ is a zero forcing set in terms of the degree sequence of $G$.  Our most general bound of this form is the following, where here and throughout $d(v)$ denotes the degree of $v$ in the graph $G$.

\begin{lem}\label{lem:degreeLower}
Let $G$ be an $n$-vertex graph with at least one edge and $p\in [0,1]$. Then
\[\Pr[\Bzf{p}{G} ]\le\sum_{v\in V(G)} d(v) p^{d(v)}.\]
\end{lem}

\begin{proof}
Let $A$ be the event that $B_p(G)=V(G)$.  For any $v\in V(G)$, let $F_v$ be the event that $v$ and exactly $d(v)-1$ of its neighbors are in $B_p(G)$.  We claim that for $B_p(G)$ to be a zero forcing set, either $A$ or $F_v$ for some $v$ must occur.  Indeed, if $B_p(G)\ne V(G)$ and $B_p(G)$ is a zero forcing set, then there must be some blue vertex $v$ in $B_p(G)$ that forces a white vertex to be blue.  In particular, if $v$ is the first vertex which performs such a force, then it and exactly $d(v)-1$ of its neighbors must be in $B_p(G)$.   This proves our claim.  Thus by the union bound we have
\[\Pr[\Bzf{p}{G}]\le \Pr[A\cup \bigcup F_v]\le \Pr[A]+\sum_{v\in V(G)} \Pr[F_v].\]
By definition, the event $F_v$ occurring means $v$ is included in $B_p(G)$ and exactly $d(v)-1$ of its $d(v)$ neighbors are included in $B_p(G)$.  As each vertex is included in $B_p(G)$ independently and with probability $p$, we have
\[\Pr[F_v]=p\cdot {d(v)\choose d(v)-1} p^{d(v)-1}(1-p)=d(v) p^{d(v)}(1-p).\]
Plugging this into the bound above and using $\Pr[A]=p^n$ gives
\begin{equation}\Pr[\Bzf{p}{G}]\le p^n+\sum_{v\in V(G)} d(v) p^{d(v)}(1-p).\label{eq:uglyDegree}\end{equation}

By assumption, $G$ contains a vertex $u$ with $d(u)\ge 1$. For this vertex we have \[d(u)p^{d(u)}(1-p)= d(u)p^{d(u)}-d(u)p^{d(u)+1}\le d(u)p^{d(u)}-p^n,\]
where this last step used $d(u)\ge 1$ and $d(u)+1\le n$ (which always holds for $n$-vertex graphs).  Plugging this bound into \eqref{eq:uglyDegree}, and using the bound $d(v)p^{d(v)}(1-p)\le d(v)p^{d(v)}$ for every other term of the sum gives the desired result.
\end{proof}

We also make use of the following, which can be proven using calculus.
\begin{obs}\label{obs:calculus}
If $d$ is a positive integer and $p\le e^{-1/d}$, then
\[\max_{x\ge d} x p^x=dp^d.\]
\end{obs}

This result quickly gives Theorem~\ref{thm:degree}, which we restate below.

\begin{thmn}[\ref*{thm:degree}]
Let $G$ be an $n$-vertex graph with minimum degree at least $\del\ge 1$.  For all $p$, we have
\[\Pr[\Bzf{p}{G}]\le \del n p^\del.\]
\end{thmn}

\begin{proof}
When $G=K_2$ the theorem is equivalent to $2p(1-p)+p^2\le 2p$, i.e. that $-p^2\le 0$, so the result holds.  From now on we assume $G$ has at least 3 vertices.  For all $p$ and $n\ge 3$, we have
\[\Pr[\Bzf{p}{G}] \le e^{-1} \del n\]
since $e^{-1}\del n\ge 1$.  This implies the result when $p\ge e^{-1/\del}$.

Observation~\ref{obs:calculus} together with Lemma~\ref{lem:degreeLower} and the fact that $d(v)\ge \delta$ for all $v$ gives
\[\Pr[\Bzf{p}{G}]\le\sum_{v\in V(G)} d(v) p^{d(v)} \le \del n p^{\del}.\qedhere\]
\end{proof}

We next prove a slightly stronger version of this theorem which holds for graphs with ``few'' vertices of degree less than a given degree $d$.

\begin{thm}\label{thm:degreeLower}
Let $G$ be an $n$-vertex graph without isolated vertices.  Suppose that there exist integers $1\le d\le n$ and $N\ge 0$ such that $G$ contains at most $N^k$ vertices of degree $k$ for all $1\le k<d$. Then for all $p\le e^{-1/d}$, we have
\[\Pr[\Bzf{p}{G}]\le 4pN+dnp^d.\]
\end{thm}

\begin{proof}
The result is trivial if $pN>\half$, so we can assume $pN\le \half$.  
Using Observation~\ref{obs:calculus} together with Lemma~\ref{lem:degreeLower} and the assumptions on $G$, we find
\begin{align*}
   \Pr[\Bzf{p}{G}]&\le\sum_{v\in V(G)} d(v) p^{d(v)} \le \sum_{\substack{v\in V(G)\\ d(v)<d}} d(v) p^{d(v)}+\sum_{\substack{v\in  V(G)\\ d(v)\ge d}} d p^d\\&\le \sum_{k=1}^{d-1} k(pN)^k+dnp^d\le \sum_{k=1}^\infty k(pN)^k+dnp^d.
\end{align*}
Note that in general we have
$\sum_{k=1}^\infty kc^k=\f{c}{(c-1)^2}$
provided $|c|<1$.  Applying this with $c=pN\le \half$ gives the desired result.
\end{proof}

We now prove analogs of these results for $z(G;k)$.
\begin{lem}\label{lem:degreeCount}
Let $G$ be an $n$-vertex graph with at least one edge. Then for all non-negative integers $k$,
\[z(G;k)\le \sum_{v\in V(G)} d(v){n-d(v)\choose k-d(v)}.\]
\end{lem}
\begin{proof}
The result is trivial if $k=n$.  For $k<n$, every zero forcing set $S$ must contain some vertex $v$ of positive degree and exactly $d(v)-1$ of its neighbors in order to have a vertex force.  Thus every zero forcing set of size $k$ can be constructed by first including a vertex $v$, then including exactly $d(v)-1$ of its neighbors, then arbitrarily including $k-d(v)$ additional vertices.  In total the number of ways to construct such a set is
\[\sum_{v\in V(G)} {d(v)\choose d(v)-1}{n-d(v)\choose k-d(v)}=\sum_{v\in V(G)} d(v){n-d(v)\choose k-d(v)},\]
so $G$ has at most this many zero forcing sets of size $k$.
\end{proof}
We next need the following lower bound on $z(P_n;k)$.
\begin{lem}\label{lem:pathCount}
For all non-negative integers $k$ we have
\[z(P_n;k)\ge \f{k^2}{n+k^2}{n\choose k}.\]
\end{lem}
\begin{proof}
Recall that (\ref{eqn:zfs(path)}) states $z(P_n;k)={n\choose k}-{n-k-1\choose k}$ for all $k$. Observe that
\[
\begin{split}
    \frac{{n-k-1\choose k}}{{n\choose k}}=&\f{(n-k-1)(n-k-2)\cdots (n-2k)}{n
    (n-1)\cdots (n-k+1)}
    =\prod_{i=0}^{k-1} \left(1 - \f{k+1}{n - i} \right)\\
    \le &\left(1-\f{k}{n}\right)^k\le \f{1}{1+k^2/n}
    =\f{n}{n+k^2},\end{split}
\]
where this last inequality follows from the Bernoulli inequality; see e.g.,~\cite[Equation $(r'_5)$]{li2013some}.   This implies
\[z(P_n;k)={n\choose k}-{n-k-1\choose k}\ge \left(1-\f{n}{n+k^2}\right){n\choose k}=\f{k^2}{n+k^2}{n\choose k}.\qedhere
\]
\end{proof}

We now prove Proposition~\ref{prop:degree}, which we restate below.

\begin{propn}[\ref*{prop:degree}]
Let $G$ be an $n$-vertex graph with minimum degree $\del\ge 3$ and $k\le (2\del)^{-1/\del}n^{1-1/\del}$. Then $z(G;k)\le z(P_n;k)$.
\end{propn}
\begin{proof}
By (\ref{eqn:zfs(path)}), for $k\ge n/2$ we have $z(P_n;k)={n\choose k}$. 
Thus we may assume throughout that $k\le n/2$.

Observe that for all $t$,
\[{n-t\choose k-t}/{n\choose k}=\f{k(k-1)\cdots (k-t+1)}{n(n-1)\cdots (n-t+1)}\le (k/n)^t,\]
with this last step using that  $(k-i)/(n-i)\le k/n$ for $i\ge 1$ if and only if $k\le n$.   Using this and Lemma~\ref{lem:degreeCount}  gives
\[z(G;k)\le \sum_v d(v)(k/n)^{d(v)}{n\choose k}.\]
Because $\del\ge 3$, we have $k\le e^{-1/\del} n$, so by Observation~\ref{obs:calculus} we have
\begin{equation}z(G;k)\le \del n (k/n)^\del {n\choose k}.\label{eq:whatever}\end{equation}

First consider the case $k\le \sqrt{n}$.  By \eqref{eq:whatever} and Lemma~\ref{lem:pathCount}, to prove $z(G;k)\le z(P_n;k)$, it suffices to have $\del n (k/n)^\del\le k^2/2n$, or equivalently $n/k\ge (2\del)^{1/(\delta-2)}$. Since $k\le \sqrt{n}$, it suffices to prove $n\ge (2\del)^{2/(\del-2)}$, and this is true for $3\le \del \le n$ and $n\ge 5$.
Thus we may assume $k\ge \sqrt{n}$.  In this case \eqref{eq:whatever} and Lemma~\ref{lem:pathCount} imply $z(G;k)\le z(P_n;k)$ provided
\[\del n(k/n)^{\del}\le \half,\]
and this will hold precisely when $k\le (2\del)^{-1/\del}n^{1-1/\del}$.
\end{proof}


With Proposition~\ref{prop:degree} we can prove Corollary~\ref{cor:degree}, which we restate below.

\begin{corn}[\ref*{cor:degree}]
Let $G$ be an $n$-vertex graph with minimum degree $\del\ge \log_2(n)+2\log_2\log_2(n)$.  Then  $z(G;k)\le z(P_n;k)$ for all $k$.
\end{corn}

\begin{proof}
The result trivially holds for $k\ge n/2$, so it suffices to prove the result for $k\le n/2$.  By Proposition~\ref{prop:degree}, it suffices to show

\[n/2\le (2\del)^{-1/\del}n^{1-1/\del},\]
or equivalently $n\le 2^\del (2\del)^{-1}$.  And indeed, for $n\ge 9$ the minimum degree condition implies
\[2^\del (2\del)^{-1} \ge n\frac{(\log_2(n))^2}{2(\log_2(n)+2\log_2\log_2(n))} \ge n.
\]
For $n\le 8$ one can check that $\lceil \log_2(n)+2 \log_2\log_2(n)\rceil\ge n-1$, so our hypothesis on $\delta$ implies $G$ is complete and the result is immediate.  In either case we conclude the result.
\end{proof}

\section{Bounds on Threshold Probabilities}\label{Sec:path}

In this section we prove that for any $n$-vertex graph $G$, the threshold probability $p(G)$ is asymptotically at least that of $P_n$, i.e.\ $p(G)=\Omega(n^{-1/2})$.  At a high level, our proof revolves around finding a graph $\tilde{G}$ which has minimum degree 2 and $p(\tilde{G})\approx p(G)$.  Because $\tilde{G}$ has minimum degree 2, Theorem~\ref{thm:degree} implies $p(G)\approx p(\tilde{G})=\Omega(n^{-1/2})$.
 We begin with a preliminary result regarding graphs containing pendant paths.

We say that a path $v_1\cdots v_k$ in a graph $G$ is a \textit{pendant path}\footnote{Most authors do not impose any conditions on the degree of $v_k$ in the definition of a pendant path, but this formulation will be more useful to us.} provided $k\ge 2$, $d_G(v_1)=1$, $d_G(v_i)=2$ for $1<i<k$, and $d_G(v_k) > 2$. We refer to the vertex $v_1$, the vertex of degree one, as the \textit{pendant vertex}, and to $v_k$, the vertex of degree at least 3, as the \textit{anchor vertex}. Observe that the only tree that does not contain a pendant path is the path graph.

\begin{lem}\label{w pend paths}
Let $G$ be an $n$-vertex graph.  If there exists a vertex $w\in V(G)$ that is the anchor of two distinct pendant paths in $G$, then $p(G)=\Om(n^{-1/2})$.
\end{lem}
\begin{proof}
Let $w\in V(G)$ and assume that $w$ is the anchor of two distinct pendant paths in $G$, i.e.,\ there exist distinct pendant paths $u_1\cdots u_k w$ and $u_{k+1}\cdots u_\ell w$ in $G$. Let
\[
I = \{j\in \Z: 1< j <k \text{ or } k+1< j < \ell\}
\]
and for each $i\in I$, let $A_i$ be the event that $u_i,u_{i+1}\in B_p(G)$. Let $A'$ be the event that $B_p(G)\cap \{u_1,u_k,u_{k+1},u_\ell\}\ne \emptyset$. Observe that if $B_p(G)\in \zfs(G)$, then  $A'$ or some $A_i$ event occurs. Thus,
\[
\Pr[\Bzf{p}{G}] \leq \Pr\Big[\bigcup_{i\in I} A_i \cup A'\Big] \leq \Pr[A'] + \sum_{i\in I}\Pr[A_i] \leq 4p + np^2.
\]
Thus to have $\Pr[\Bzf{p}{G}] = \half$, we must have $np^2+4p\ge \half$, which implies $p=\Om(n^{-1/2})$.
\end{proof}

With this lemma, we see that when proving Theorem \ref{thm:pathThreshold} we may assume each vertex is the anchor of at most one pendant path. The next lemma allows us to assume that none of these paths are too long (unless $G$ consists of a single path). In order to prove the next lemma, we recall various definitions and notation related to forcing chains which can be found, for example, in \cite{BBF10}.

Let $G$ be a graph and $B\subseteq V(G)$. Using $B$ as the initial set of blue vertices, apply the color change rule and record the forces. If a vertex $v$ forces $u$ we write $v\to u$. The \textit{chronological list of forces} $\mathcal{F}$ is the ordered list of forces, written in the order they were performed, that produces the final coloring of $B$ in $G$. We shall sometimes use $\mathcal{F}$ to denote the \textit{set of forces} that produces the final coloring of $B$ in $G$. A \textit{forcing chain} of $\mathcal{F}$ is a sequence of vertices $(v_1,\ldots,v_k)$ such that $v_i\to v_{i+1}$ for $1\leq i \leq k-1$. A \textit{maximal forcing chain} of $\mathcal{F}$ is a forcing chain that is not a proper subsequence of another forcing chain of $\mathcal{F}$. The \textit{reversal} of $B$ for $\mathcal{F}$ is the set of all vertices that do not perform a force, i.e., the set of all vertices that are the last element in a maximal forcing chain of $\mathcal{F}$.

\begin{lem}\label{add a clique new}
Let $G$ be an $n$-vertex graph and let $\{v_1,\ldots,v_M\}$ denote a set of vertices of degree 1 in $G$. Let $\tilde{G}$ be the graph obtained from $G$ by adding a clique on $\{v_1,\ldots,v_M\}$.  Then
\[
\Pr[\Bzf{p}{G}]\le \Pr[\Bzf{p}{\tilde{G}}] + pM.
\]
\end{lem}

\begin{proof}
We begin by showing that if $B\in\zfs(G)$ and $B\notin\zfs(\tilde{G})$, then $v_i\in B$ for some $i = 1,\ldots, M$. Let $B\subseteq V(G)$ and suppose that $v_i\notin B$ for all $i$. Assume that $B\in\zfs(G)$ and let $\mathcal{F}$ be the set of forces for $B$ in $G$. Since each $v_i$ is a pendant vertex in $G$ and each $v_i\notin B$, the set $\{v_1,\ldots,v_M\}$ is contained in the reversal of $B$ for $\mathcal{F}$. This, and the fact that the neighborhood of each vertex in $V(G) \setminus \{v_1,\ldots,v_M\}$ is unchanged by adding a clique to $\{v_1,\ldots,v_M\}$, implies $\mathcal{F}$ is a set of forces for $B$ in $\tilde{G}$. Thus, $B\in\zfs(\tilde{G})$ and hence we have shown that if $B\in\zfs(G)$ and $B\notin\zfs(\tilde{G})$, then $v_i\in B$ for some $i = 1,\ldots, M$.

Let $A_i$ be the event that $v_i\in B_p(G)$. By the preceding argument and the union bound,
\[
\Pr[B_p(G)\in\zfs(G) \text{ and } B_p(G)\notin\zfs(\tilde{G})] \leq \Pr\left[\cup_{i=1}^M A_i\right]
\leq pM.
\]
We also have
\begin{align*}
\Pr[&B_p(G)\in\zfs(G) \text{ and } B_p(G)\notin\zfs(\tilde{G})] \\
=&
\Pr[B_p(G)\in\zfs(G)] - \Pr[B_p(G)\in\zfs(G) \text{ and } B_p(G)\in\zfs(\tilde{G})]\\
\geq& \Pr[B_p(G)\in\zfs(G)] - \Pr[B_p(G)\in\zfs(\tilde{G})].
\end{align*}
Combining both inequalities gives the result.
\end{proof}

The \textit{$2$-core} of a graph $G$, denoted $C_2(G)$, is obtained from $G$ by repeatedly removing all isolated vertices and all vertices of degree 1 from $G$ until no further removals are possible. See \cite{Bickel2010the} for basic facts about $2$-cores.
We say that $T$ is a \textit{pendant tree} of a graph $G$ if $T$ is a maximal induced subgraph of $G$ such that $T$ is a tree, and if there exists a unique vertex $w\in V(T)$ contained in $C_2(G)$. The vertex $w$ is called the \textit{anchor vertex} of $T$. It is known that a vertex $v$ is in $C_2(G)$ if and only if $v$ is contained in a cycle or a path between cycles. Thus, the $2$-core of $G$ can be obtained by removing all non-anchor vertices of each pendant tree  and all components of $G$ that are trees.

Let $G$ be a graph and $B\subseteq V(G)$. We define $C_2(B,G)$ to be the set of vertices that are either contained in $B\cap C_2(G)$ or are anchor vertices of a pendant tree $T$ such that $B\cap V(T)$ is a zero forcing set of $T$.  When $G$ is clear from context we simply write $C_2(B)$. These definitions are illustrated in Figure \ref{fig:C2(G)}. The motivation for these definitions is found in Lemma \ref{without pendants}.

\begin{figure}[h!]
    \centering
    \begin{subfigure}[b]{0.4\textwidth}
         \centering
         \includegraphics[scale=0.4]{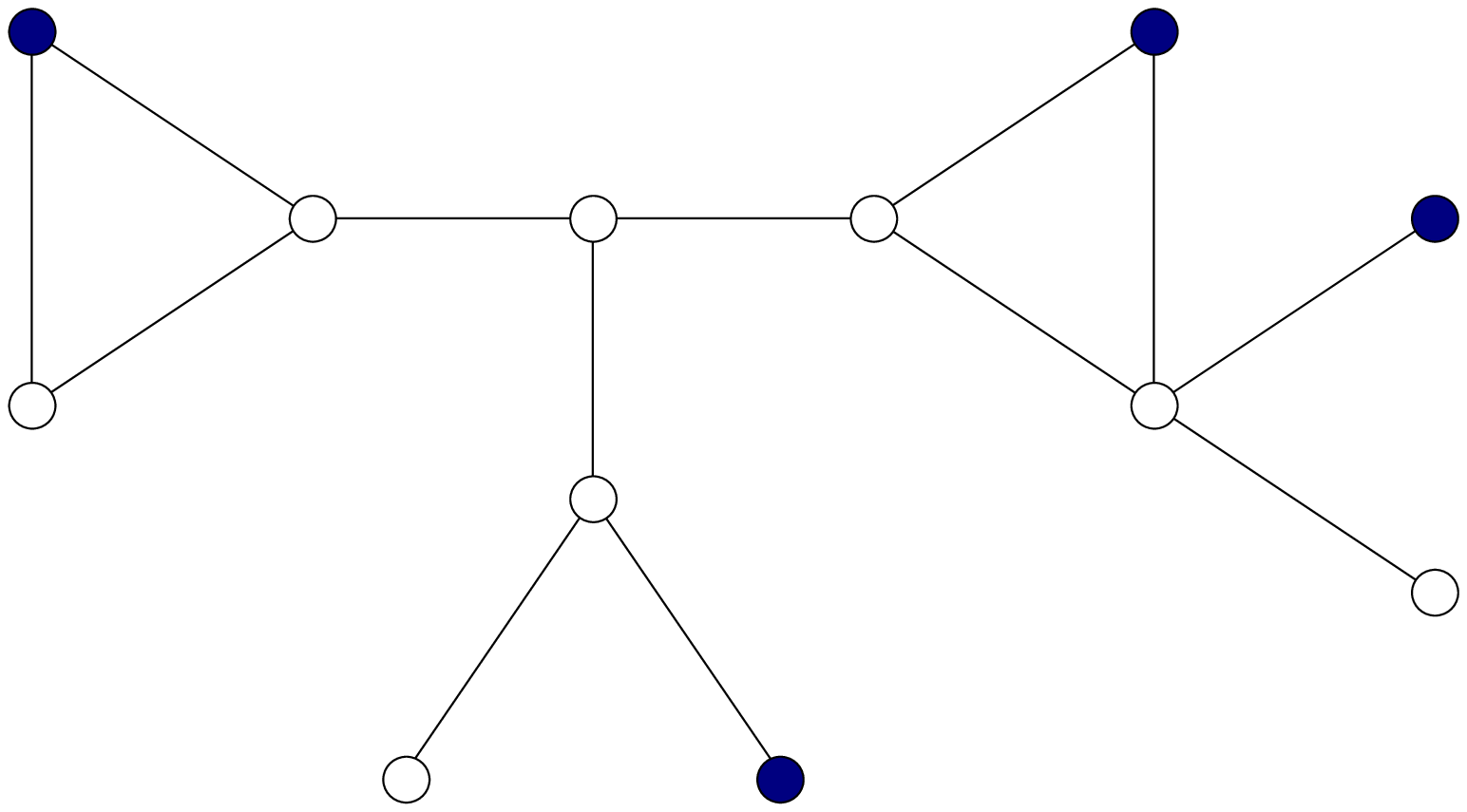}
         \caption{}
         \label{fig:G_B}
     \end{subfigure}
        \hfil\hfil\hfil
     \begin{subfigure}[b]{0.4\textwidth}
         \centering
         \includegraphics[scale=0.4]{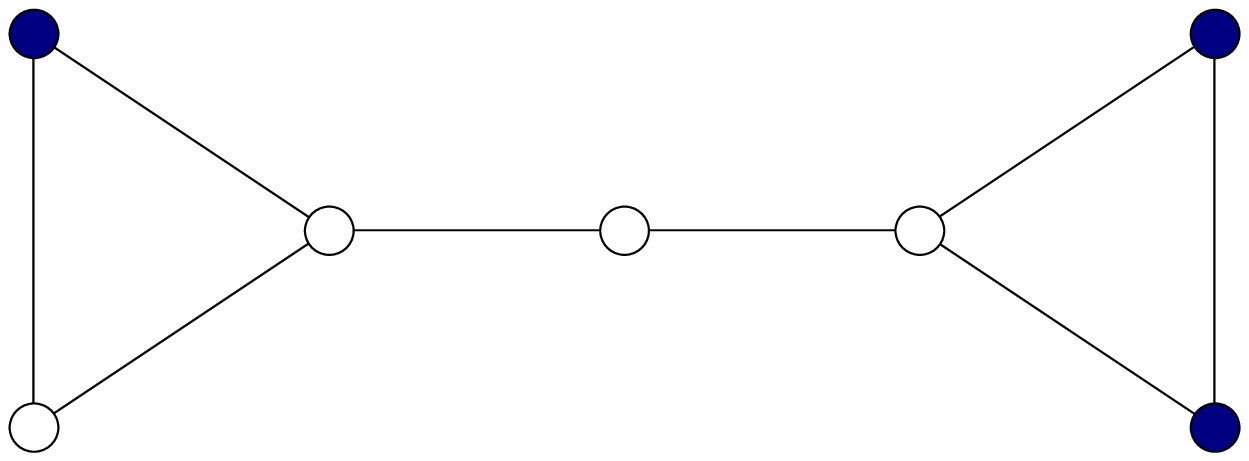}
         \vspace{46pt}
         \caption{}
         \label{fig:G>2_B>2}
     \end{subfigure}
    \caption{(a) Graph $G$ with zero forcing set $B$ colored blue. (b) Graph $C_2(G)$ with zero forcing set $C_2(B)$ colored blue.}
    \label{fig:C2(G)}
\end{figure}

Before proving our next lemma, we note the following observation about zero forcing on graphs with a cut vertex. Observation \ref{cut vertex obs} follows from some of the concepts introduced in \cite{D12}. We write $G[S]$ to denote the induced subgraph of the graph $G$ on $S\subseteq V(G)$.
\begin{obs}\label{cut vertex obs}
Let $G$ be a graph with cut vertex $w$, and let $W_1\cup W_2\cup \{w\}$ be a partition of $V(G)$ such that $W_1$ and $W_2$ are the disjoint union of connected components of $G-w$. Let $G_i = G \big[V(W_i) \cup\{w\}\big]$ for $i = 1,2$ . Then $B\cap V(G_i)\in \zfs(G_i)$ for some index $i$, and $B\cap V(G_i)\cup\{w\}$ is a zero forcing set of $G_i$ for each $i=1,2$.
\end{obs}

\begin{lem}\label{without pendants}
Let $G$ be a graph and $B\subseteq V(G)$. If $B$ is a zero forcing set of $G$, then $C_2(B)$ is a zero forcing set of $C_2(G)$.
\end{lem}
\begin{proof}

Assume for contradiction that there exists a pair $(B,G)$ such that $B$ is a zero forcing set of $G$ and $C_2(B)$ is not a zero forcing set of $C_2(G)$.  Moreover, choose a minimal counterexample $(B,G)$ such that $G$ has as few vertices as possible.

We begin with a few observations to simplify the proof. The 2-core of $G$ is the disjoint union of the 2-cores of each connected component of $G$. If $C_2(G)$ is the null graph, then it is vacuously true that $C_2(B)$ is a zero forcing set of $C_2(G)$. If $C_2(G) = G$, then $C_2(B) = B$ and hence $C_2(B)$ is a zero forcing set of $C_2(G)$. We may therefore assume that $G$ is connected and that $G$ contains a pendant tree.

Let $T$ be a pendant tree of $G$ with anchor vertex $w$, and let $G_T$ be the induced subgraph of $G$ on $(V(G)\setminus V(T))\cup\{w\}$. Let
\[
B_T =
\begin{cases}
B \cap V(G_T) & \text{if } B\cap V(T)\notin \zfs(T)\\
(B \cap V(G_T))\cup \{w\} & \text{if } B\cap V(T)\in \zfs(T).
\end{cases}
\]
Since $w$ is a cut vertex, Observation \ref{cut vertex obs} implies that $B_T$ is a zero forcing set of $G_T$.  By our assumption of $(B,G)$ being a vertex minimal counterexample, we have that $C_2(B_T,G_T)$ is a zero forcing set of $C_2(G_T)$ since $G_T$ has strictly fewer vertices than $G$.  It is not difficult to check that $C_2(G_T) = C_2(G)$ and $C_2(B_T,G_T) = C_2(B,G)$, so $C_2(B,G)$ is a zero forcing set of $C_2(G)$.  This gives a contradiction, proving the result.
\end{proof}

We can now prove the main result of this section, which we restate below.
\begin{thmn}[\ref*{thm:pathThreshold}]
If $G$ is an $n$-vertex graph, then
\[\thresh{G}=\Om(\thresh{P_n})=\Om(n^{-1/2}).\]
\end{thmn}
\begin{proof}
Observe that if $H$ is a connected component of $G$, then $p(G)\geq p(H)$. Also, by Lemma \ref{w pend paths}, if $G$ is a tree that is not a path, or if $G$ has a pendant tree that is not a pendant path, then $p(G) = \Omega(n^{-1/2})$. We may therefore assume that $G$ is connected, $G$ contains a cycle, and every pendant tree of $G$ is a pendant path. Note that the definition of pendant trees implies every vertex $v$ is the anchor of at most one pendant path in $G$, and that every anchor vertex is contained in $C_2(G)$. Let $p = cn^{-1/2}$, where $c<1$ is a positive constant which we specify later.

Let $\{P_1,\ldots,P_M\}$ be the set of all pendant paths in $G$ on at least $100n^{1/2}+1$ vertices, and let $v_i\in P_i$ denote the vertex of $P_i$ of degree 1.  Let $\tilde{G}$ be the graph obtained from $G$ by adding a clique on $\{v_1,\ldots,v_M\}$. Observe that $M(100n^{1/2} + 1)\le n$, and hence $M < n^{1/2}/100$. Thus by Lemma \ref{add a clique new},
\[
\Pr[\Bzf{p}{G}] < 
\Pr[\Bzf{p}{\tilde{G}}] + .01.
\]
Observe that $C_2(\tilde{G})$ is nonempty since $G$ contains a cycle, and by Lemma \ref{without pendants},
\[
\Pr[B_p(\tilde{G}) \in \zfs(\tilde{G})] \leq \Pr[C_2(B_p(\tilde{G})) \in \zfs(C_2(\tilde{G}))].
\]
Thus it suffices to prove $\Pr[C_2(B_p(\tilde{G})) \in \zfs(C_2(\tilde{G}))] < 0.49$ for $n$ sufficiently large.

For $v\in V(C_2(\tilde{G}))$, let $A_v$ denote the event that $v\in C_2(B_p(\tilde{G}))$.  We claim that
\begin{equation}\Pr[A_v]\le  2p + (100n^{1/2})p^2 = (2c + 100c^2)n^{-1/2}:=q.\label{eq:q}\end{equation}
Indeed, $\Pr[A_v]=p$ if $v$ is not the anchor of some pendant path. If $v$ is the anchor of the pendant path $P$, then for $A_v$ to occur, either an endpoint of $P$ or two consecutive vertices of $P$ must be in $B_p(\tilde{G})$, and a union bound gives the result since $P$ is assumed to have length at most $100n^{1/2}+1$.

Because the $A_v$ events are independent of each other, our bound above implies that
\[\Pr[C_2(B_p(\tilde{G})) \in \zfs(C_2(\tilde{G}))]\le \Pr[B_q(\tilde{G})\in \zfs(C_2(\tilde{G})).\]
Since $C_2(\tilde{G})$ has at most $n$ vertices and minimum degree at least 2, Theorem \ref{thm:degree} implies
\[
\Pr[B_q(\tilde{G}) \in \zfs(C_2(\tilde{G}))] \leq 2 n q^2.
\]
Taking $c= 1/17$ and recalling the definition of $q$ in \eqref{eq:q} gives the desired result.
\end{proof}


\section{Bounds for Trees}\label{Sec:tree}
In this section we prove $\Pr[B_p(T)\in\zfs(T)]\le \Pr[B_p(P_n)\in\zfs(P_n)]$ whenever $T$ is an $n$-vertex tree with $n$ sufficiently large.  We will break our proof into two cases, namely when $p=\Omega(n^{-1})$ and $p=O(n^{-1})$.  The intuition for this choice is that when $p\ll n^{-1},$ the probability that $B_p(P_n)$ is zero forcing is roughly  the probability of choosing an endpoint, while for $p\gg n^{-1}$ it is roughly the probability of choosing two consecutive vertices.  As such, we will need two different arguments for these two regimes.

\subsection{Large $p$}
The following provides a concrete statement agreeing with the intuition outlined above.
\begin{lem}\label{large p: endpoints less likely than consecutive pair}
Let $v_1,\ldots,v_n$ be vertices of a graph $G$.  If $\frac{8}{n}<p<1$ and $n\geq 16$, then \[\Pr[v_1\in B_p(G)\tr{ or }v_n\in B_p(G)]<\Pr[v_i,v_{i+1}\in B_p(G)\tr{ for some }i].\]
\end{lem}
\begin{proof}

Define
$$p_{e} := \Pr[v_1\in B_p(G)\tr{ or }v_n\in B_p(G)]  = 1-(1-p)^2,$$
$$p_m := \Pr[v_i,v_{i+1}\in B_p(G)\tr{ for some }i] \geq 1 - (1-p^2)^{\lfloor \frac{n-1}{2}\rfloor},$$
where this last inequality holds because $p_m$ is strictly more than the probability of having at least one of the pairs $(v_i, v_{i+1})$ with $i$ odd in $B_p(G)$; see the proof of Proposition~\ref{prop:p(P_n)} for a more formal argument.

We will show that when $n \geq \frac{4}{p-p^2} $, we have $p_m > p_e$.  Indeed, $n \geq \frac{4}{p-p^2}$ is equivalent to $\frac{-2p}{(1-p)} \geq \frac{-p^2n}{2}$.
By looking at the Taylor series, it is easy to show $\frac{x}{1+x} < \ln(1+x) < x$ for $|x|<1$, so we have

$$\frac{-2p}{(1-p)} < 2\ln(1-p)\ \text{ and }\  \ln(1-p^2)\cdot \floor{\frac{n-1}{2}}<-p^2\floor{\frac{n-1}{2}} <\frac{-p^2n}{2}.$$

Thus for this range of $n$ we have $2\ln(1-p) > \floor{\frac{n-1}{2}}\ln(1-p^2)$, or equivalently,
$$1-(1-p)^2 <1- (1-p^2)^{\floor{\frac{n-1}{2}}}.$$ Thus $p_m > p_e$ when $n \geq \frac{4}{p-p^2}$.
This bound on $n$ holds for all $p\in (\frac{8}{n}, 1-\frac{8}{n})$, and for $p \in [1 - \frac{8}{n}, 1)$, this same bound holds for $n\geq 16$, completing the proof.
\end{proof}
Analogous to Proposition~\ref{w pend paths}, we can show that graphs with a vertex at the end of two short pendant paths are harder to zero force than paths.

\begin{lem}\label{large p: paths are easier than graphs with joined pendant paths}
	Let $G$ be an $n$-vertex graph that has a vertex $w$ which is the endpoint of two pendant paths $u_1\cdots u_sw$ and $v_1\cdots v_tw$.  If $p \geq 8/(n-s-t)$ and $n-s-t \geq 14$, then $$\Pr[B_p(G)\in\zfs(G)]<\Pr[B_p(P_n)\in \zfs(P_n)].$$
\end{lem}
\begin{proof}
	Let $w_1,\ldots,w_r$ be an arbitrary ordering of $V(G)\sm\{u_1,\ldots,u_s,v_1,\ldots,v_t\}$.  Relabel the vertices of $P_n$ so that its vertices along the path are $u_1\cdots u_s w_1\cdots w_r v_t\cdots v_1$.  Because $V(G)=V(P_n)$, we can couple our random variables so that $B_p:=B_p(G)=B_p(P_n)$.   Let $F=B_p\cap \{u_1,\ldots,u_{s-1},v_1,\ldots,v_{t-1}\}$.  It suffices to show for all $S\sub \{u_1,\ldots,u_{s-1},v_1,\ldots,v_{t-1}\}$ that
	\[\Pr[B_p(G)\in \zfs(G)|F=S]\le \Pr[B_p(P_n)\in \zfs(P_n)|F=S],\]
	with strict inequality for at least one such set.  If $S$ contains two consecutive vertices $u_i$ and $u_{i+1}$, two consecutive vertices $v_j$ and $v_{j+1}$, or $u_1$ or $v_1$, then $B_p\in \zfs(P_n)$ so the result holds trivially.  Thus from now on we can  assume this is not the case.

	With this assumption, the vertex $u_s$ in $G$ can only be colored blue by $B_p$ if at least one of $u_s$ or $v_t$ is in $B_p$ (this is because $u_s$ is adjacent to $u_{s-1}$, which does not enact any forces by assumption on $S$, and to $w$, which can only enact a force if at least one of $u_s,v_t$ are colored blue at some point). On the other hand, $B_p$ will be a zero forcing set for $P_n$ provided $B_p$ contains two consecutive vertices from $\{u_s, w_1, \cdots, w_r, v_t\}$.

	By applying Lemma~\ref{large p: endpoints less likely than consecutive pair} to the vertex set $\{u_s,w_1,\ldots,w_r,v_t\}$, we see that if $p>\frac{8}{n-s-t+2}$ and $n-s-t+2\ge 16$, then
	$$ \Pr[B_p \text{ contains $u_s$ or $v_t$}]<\Pr[B_p \text{ contains a consecutive pair from } \{u_s, w_1, \cdots, w_r, v_t\}
	].$$
	As these two events are independent of the random set $F$, we conclude that for $p \geq 8/(n-s-t) > 8/(n-s-t+2)$ and $n-s-t+2 \geq 16$,
	$$ \Pr[B_p(G)\in\zfs(G)|F=S]<\Pr[B_p(P_n)\in \zfs(P_n)|F=S],$$
	and from this we conclude the result.
\end{proof}
With this we can solve Theorem~\ref{thm:tree} when $p=\Omega(n^{-1})$.
\begin{prop}\label{prop:tree1}
	If $T\ne P_n$ is an $n$-vertex tree with $n\ge 42$, and if $ \frac{24}{n}<p<1$, then
	\[\Pr[B_p(T)\in\zfs(T)] <\Pr[B_p(P_n)\in \zfs(P_n)].\]
\end{prop}

\begin{proof}
	Let $u,v$ be any two leaves of $T$ which are at a shortest distance from each other.  Observe that the path between $u,v$ consists of two pendant paths, say $uu_2, \cdots u_sw$ and $vv_2, \cdots v_tw$.  Because $T$ is not a path, there either exists exactly one leaf $\ell\ne u,v$, or at least two leaves $i,j\ne u,v$.

	Suppose for contradiction that  $s + t > \frac{2n}{3}$.  If $T$ has exactly three leaves, then \[d(\ell,u)=d(\ell,w)+s< n/3+s,\]
	where this inequality used that none of the internal vertices along the path from $\ell$ to $w$ use any of the vertices along the path from $u$ to $v$, of which there are more than $2n/3+1$ vertices.  By a symmetric argument we have $d(\ell,v)<n/3+t$.  In particular, we must have \[d(\ell,u)+d(\ell,v)<2n/3+s+t<2(s+t)=2d(u,v),\]
	and hence at least one of $d(\ell,u),d(\ell,v)$ is smaller than $d(u,v)$, a contradiction to our choice of $u,v$.  Similarly if $T$ has at least four leaves, then  $d(i,j) < \frac{n}{3}$, which again gives a contradiction.  Thus we can assume $s+t\le 2n/3$. With this, our hypothesis implies $p \geq \f{8}{n-s-t}$ and $n-s-t+2 \geq 16$, so we can apply Lemma \ref{large p: paths are easier than graphs with joined pendant paths} to give the desired result.
\end{proof}

\subsection{Small $p$}
We will prove our result for small $p$ by upper bounding $z(T;k)$, which we recall is the number of zero forcing sets of $T$ of size $k$.

\begin{lem}\label{lem:leavesMaxDeg}
If $T\ne P_n$ is an $n$-vertex tree, then \[z(T;k)\le \frac{13 k^4}{n^2}{n\choose k}.\]
\end{lem}

\begin{proof}
Let $\Del$ denote the maximum degree of $T$ and $\ell$ the number of leaves of $T$.  Observe that the zero forcing number of a graph is always at least the minimum number of paths needed to cover the vertices of the graph.  In particular, every zero forcing set for the tree $T$ has size at least $\ell/2$.  It is also known (see for example \cite{O21}) that every zero forcing set for a tree $T$ has size at least $\Del-1$. Thus for $k<\max\{\Del-1,\ell/2\}$, we have $z(T;k)=0$ and the bound trivially holds.  From now on we assume $\Del-1,\ell/2\le k$.  The bound is also trivial when $k=n$, so we may assume $k<n$.  Lastly, we may also assume $\Del\ge 3$ since $T$ is not a path.

We first count the number of zero forcing sets $S$ of size $k$ which have two pairs of vertices $u,v$ and $x,y$ with $u\sim v,\ x\sim y$ and $\{u,v\}\cap \{x,y\}=\emptyset$.  In this case the number of sets $S$ is at most
\[(n-1)^2 {n-4\choose k-4},\]
since one can choose each pair (which is just an edge in $T$) in at most $n-1$ ways.

We next count the number of zero forcing sets $S$ of size $k$ which contain three vertices $u,v,w$ with $u\sim v\sim w$.  The number of such $S$ is at most \[(n-1)(2\Del-2){n-3\choose k-3},\] since one can first choose two adjacent vertices in $n-1$ ways, then a third vertex which is adjacent to at least one of these in at most $2\Del-2$ ways, and then the remaining vertices in ${n-3\choose k-3}$ ways.

We next count the number of $S$ that contain no two adjacent vertices.  Because $k<n$ and $S$ is a zero forcing set, at least one vertex of $S$ must be able to force.  Because $S$ contains no adjacent vertices, this is only possible if $S$ contains a leaf.  Choose such a leaf $u_1$ to include in $S$, which can be done in $\ell$ ways.  Let $u_1u_2\cdots u_s$ be the unique path in $T$ with $\deg(u_i)=2$ for $1<i<s$ and $\deg(u_s)\ne 2$.
\begin{claim}
The set $S$ either contains a leaf $v\ne u_1$, or a neighbor of $u_s$ other than $u_{s-1}$.
\end{claim}
\begin{proof}
Assume this were not the case.  Because $S$ contains no two adjacent vertices, no additional leaves, and no other neighbor of $u_s$, it is not difficult to see that the only vertices that will be colored blue by $S$ are $S\cup \{u_2,\ldots,u_s\}$.  Because $S$ is a zero forcing set, we must have $V(T)=S\cup \{u_2,\ldots,u_s\}$.  However, by assumption the only leaves that could be in $S\cup \{u_2,\ldots,u_s\}$ are $u_1$ and $u_s$, but $T\ne P_n$ contains at least three leaves, so  $V(T)\ne S\cup \{u_2,\ldots,u_s\}$, giving the desired contradiction.
\end{proof}
In total then, we see that the number of choices for such a set $S$ is at most
\[\ell(\ell+\Del-1){n-2\choose k-2},\]
where the terms in the expression above count the number of choices for $u_1$, followed by the number of choices for some additional leaf or neighbor of $u_s$, followed by the number of arbitrary sets of $k-2$ vertices.

It remains to count $S$ that have exactly one pair of adjacent vertices.  One can first choose the adjacent pair $u_1,v_1\in S$ in at most $n-1$ ways.  If $\deg(u_1)=2$, then let $u_1\cdots u_s$ be the unique path from $u_1$ with $u_2$ the neighbor of $u_1$ not equal to $v_1$, and with $\deg(u_i)=2$ for all $i<s$ and $\deg(u_s)\ne 2$.  If $\deg(u_1)\ne 2$, then we simply consider the 1 vertex path $u_1$. Analogously define the path $v_1\cdots v_t$.  As in the previous case, because $S$ contains no other pair of adjacent vertices, it must contain at least one leaf or one neighbor of either $u_s$ or $v_t$ that is not $u_{s-1}$ or $v_{t-1}$.  In total then the number of choices for such an $S$ is at most
 \[(n-1)(\ell+2\Del-2){n-3\choose k-3}.\]

In total, $z(T;k)$ is at most
{\footnotesize \begin{align*}& (n-1)^2{n-4\choose k-4}+(n-1)(2\Del-2){n-3\choose k-3}+ \ell(\ell+\Del-1){n-2\choose k-2}+(n-1)(\ell+2\Del-2){n-3\choose k-3}\\
&\le n^2\cdot \frac{k^4}{n^4}{n\choose k}+n(2\Del-2)\cdot \frac{k^3}{n^3}{n\choose k}+\ell(\ell+\Del-1)\cdot \frac{k^2}{n^2}{n\choose k}+n(\ell+2\Del-2)\cdot \frac{k^3}{n^3}{n\choose k},\end{align*}}
where this last inequality used $n-1\le n$ and that ${n-t\choose k-t}\le (k/n)^t {n\choose k}$ for all integers $t\ge 0$. Using our assumptions $\Del-1\le k$ and $\ell\le 2k$, we find that the above expression is at most $(1+2+6+4)\frac{k^4}{n^2}{n\choose k}$ as desired.
\end{proof}

With this we can prove the following.
\begin{prop}\label{prop:tree2}
For every $C>0$, there exists an integer $n_0$ such that for all $n\ge n_0$, if $T\ne P_n$ is an $n$-vertex tree and $0<p\le \f{C}{n}$, then
\[\Pr[B_p(T)\in\zfs(T)]<\Pr[B_p(P_n)\in\zfs(P_n)].\]
\end{prop}
\begin{proof}

By the previous lemma and the trivial bound ${n\choose k}\le n^k/k!$, we have

\[\Pr[B_p(T)\in\zfs(T)]\le \sum_k z(T;k)p^k\le \sum_k \frac{13 k^4 n^{k-2}}{k!} p^k\le p\cdot 13Cn^{-1}\sum_k  \frac{k^4C^{k-2}}{k!}.\]
The above sum is convergent, so for $n$ sufficiently large we find
\[\Pr[B_p(T)\in\zfs(T)]\le \frac{1}{2}p\le \Pr[B_p(P_n)\in \zfs(P_n)],\]
where this latter inequality is strict provided $p>0$.
\end{proof}
With Propositions~\ref{prop:tree1} and~\ref{prop:tree2} we can prove Theorem~\ref{thm:tree}, which we restate below.
\begin{thmn}[\ref*{thm:tree}]
If $T$ is an $n$-vertex tree with $n$ sufficiently large, then for all $0\le p\le 1$,
\[\Pr[\Bzf{p}{T}]\le \Pr[\Bzf{p}{P_n}],\]
with equality holding if and only if either $p\in \{0,1\}$ or $T=P_n$.
\end{thmn}
\begin{proof}
The equality of the result trivially holds for either $p\in \{0,1\}$ or $T = P_n$. If $T \neq P_n$ with $n$ sufficiently large, by Proposition~\ref{prop:tree1}, the result holds for $24/n<p<1$, and by Proposition~\ref{prop:tree2} the result holds for $0<p\le 24/n$.
\end{proof}

\section{Concluding Remarks}\label{Sec:remark}
We have many open problems regarding the threshold probability $\thresh{G}$, which we recall is the unique $p\in [0,1]$ such that $\Pr[\Bzf{p}{G}]=\half$.  For example, we conjecture the following refinement of Theorem~\ref{thm:pathThreshold}.
\begin{conj}\label{conj:clique}
If $G$ is an $n$-vertex graph which contains a clique of size $k$, then
\[\thresh{G}=\Om(\sqrt{k/n}).\]
\end{conj}
This conjecture can be viewed as a probabilistic analog to the classical result that $Z(G)\ge k$ if $G$ has a clique of size $k$, which was proved by Butler and Young~\cite{BY13}.  The motivation for the bound $\Om(\sqrt{k/n})$ comes from considering a graph $G$ which consists of a clique on $k$ vertices, with each of these vertices attached to a path of length roughly $n/k$.  For this graph, a given vertex of the clique will be forced by the path it is connected to with probability roughly $1-e^{-p^2n/k}$, so if $p$ is much smaller than $\sqrt{k/n}$, then almost none of the clique vertices in $G$ will be colored blue.  Thus $\thresh{G}=\Om(\sqrt{k/n})$ in this case\footnote{In fact, a sharper analysis shows that $\thresh{G}=\Om(\sqrt{k \log(k)/n})$ for $k$ not too large in terms of $n$.  We suspect that Conjecture~\ref{conj:clique} can be strengthened to include this $\log(k)$ term, but for ease of presentation we have written the conjecture as is.}.

In Section~\ref{Sec:examples}, we determined the order of magnitude of $\thresh{G}$ for many natural families of graphs.  One case where we do not know how to do this is for the $n$-dimensional hypercube $Q_n$.
\begin{prob}\label{prob:cube}
Does there exist a constant $c$ such that $\thresh{Q_n}\sim c$?  If so, what is this constant?
\end{prob}
Because $Z(Q_n)=2^{n-1}$, we must have $c\ge .5$ if it exists, but beyond this we know nothing about $c$.  The empirical plot in Figure~\ref{fig:EmpiricalGraph}
of the probability that $B_p(Q_8)$ is zero forcing
suggests that $c$ might be at least .58.   Another family of graphs which we do not understand are grid graphs.

\begin{prob}
Determine the order of magnitude of $\thresh{P_m\square P_n}$, where $P_m\square P_n$ denotes the $m\times n$ grid graph.
\end{prob}
Assuming $2\le m\le n$, we can apply Theorem~\ref{thm:degreeLower} with $d=4$ and $N\approx n^{1/3}$ to show $\thresh{P_m\square P_n}=\Om(\min\{n^{-1/3},(mn)^{-1/4}\})$.  The best general upper bound we have is $\thresh{P_m\square P_n}=O(n^{-1/2m})$, since at this point it is fairly likely that $B_p(P_m\square P_n)$ contains two consecutive $P_m$ paths, which forces the entire graph.  For small $m$ we suspect that our upper bound is closer to the truth than our lower bound, but for large $m$ the situation is unclear.

Lastly, one could consider randomized versions of variants of the classical zero forcing number.  For example, under skew zero forcing (which was originally introduced in \cite{IMA10}), one can easily generalize Theorem~\ref{thm:degree} to give an upper bound of roughly $\del np^{\del-1}$, generalizing the classical result $Z_-(G)\ge \del-1$ if $G$ has minimum degree $\delta$.  It would also be interesting to consider probabilistic zero forcing with a random set of vertices initially colored blue.

\section*{Acknowledgements}
The authors would like to express their sincere gratitude to the organizers of the Mathematics Research Community on Finding Needles in Haystacks: Approaches to Inverse Problems Using Combinatorics and Linear Algebra, and the AMS for funding the program through NSF grant 1916439.

\bibliographystyle{alpha}
\bibliography{bib}

\end{document}